\documentclass[10pt,draft,reqno]{amsart}

     \makeatletter
     \def\section{\@startsection{section}{1}%
     \z@{.7\linespacing\@plus\linespacing}{.5\linespacing}%
     {\bfseries
     \centering
     }}
     \def\@secnumfont{\bfseries}
     \makeatother
\newtheorem{theorem}{Theorem}[section]
\newtheorem{lemma}[theorem]{Lemma}
\newtheorem{proposition}[theorem]{Proposition}
\newtheorem{corollary}[theorem]{Corollary}
\theoremstyle{definition}
\newtheorem{definition}[theorem]{Definition}
\newtheorem{example}[theorem]{Example}
\theoremstyle{remark}
\newtheorem{remark}[theorem]{Remark}
\numberwithin{equation}{section}
\newtheorem{notation}[theorem]{\bf{Notation}}
\numberwithin{equation}{section}

\usepackage{amsfonts} 
\usepackage{amsthm}
\usepackage{amsmath}
\usepackage{amssymb}
\usepackage{mathrsfs}
\usepackage[utf8]{inputenc}
\usepackage[T1]{fontenc}
\usepackage{enumerate}
\newcommand*{\C}{\mathbb{C}}

\newcommand*{\N}{\mathbb{N}} 
\newcommand*{\NZ}{\mathbb{N}_0} 
\newcommand*{\R}{\mathbb{R}} 
 
\newcommand*{\halb}{\frac{1}{2}}

\def\projlim{\qopname\relax m{proj\,lim}} 
\def\indlim{\qopname\relax m{ind\,lim}}

\newcommand*{\LL}{\left\langle \left\langle} 
\newcommand*{\RR}{\right\rangle\right\rangle} 
\allowdisplaybreaks
\begin{document}
\title[Fourier Gauss transforms]{Regular transformation groups based on Fourier-Gauss transforms
}
\author[M.~Bock]{Maximilian Bock}
\address{Maximilian Bock: Saarstahl AG, Informatik, Bismarckstr. 57-59, 66333 Völklingen, Germany}
\email{maximilianbock61@gmail.com}
%
\author[W.~Bock]{Wolfgang Bock}
\address{Wolfgang Bock: Technomathematics Group,Department of Mathematics, Univer-
sity of Kaiserslautern, 67653 Kaiserslautern, Germany}
\email{bock@mathematik.uni-kl.de}
%
\date{April 15, 2015}

%
%
\begin{abstract}
We discuss different representations of the White noise spaces $(E)_\beta$, $0\leq\beta<1$ by introducing generalized Wick tensors. As an application we state the following generalization of the Mehler formula for the Ornstein-Uhlenbeck semigroup, see \cite[p. 237]{Hida1}:\\ 
Let $a,b \in \C$, $0\leq\beta<1$. Then $a\cdot\Delta_G + b\cdot N$ is the equicontinuous infinitesimal generator of the following differentiable one parameter transformation group $\left\{P_{a,b,t}\right\}_{t\in\R}\subset GL((E)_\beta)$:
\begin{enumerate}[(i)]
	\item 
	if $b\neq 0$ then for all $\varphi\in (E)_\beta,\ t\in\R:$
	\begin{equation*}	
	 P_{a,b,t}(\varphi)=\int\limits_{E^*}
	\varphi(\sqrt{(1-\frac{a}{b})(1-e^{2bt})}\cdot x + e^{bt}\cdot y)\ d\mu(x)
	\end{equation*}
	\item
		if $b=0$ then for all $\varphi\in (E)_\beta,\ t\in\R:$
	\begin{equation*}
	P_{a,0,t}(\varphi)=\int\limits_{E^*}\varphi(\sqrt{2at}\cdot x+y)\ d\mu(x)
	\end{equation*}
\end{enumerate}
On an informal level the second case of the above theorem may be looked upon as a special case of the first one. Note that by the rules of l'Hôpital we have $\lim\limits_{b\rightarrow 0} (b-a)\frac{1-e^{2bt}}{b}= 2at$.

\end{abstract}

\maketitle

\section{Introduction}
%
\noindent
In finite dimensional analysis convolution and Fourier transform appear in many different applications. Since the Fourier transform is symmetric w.r.t.~dual pairing it extends naturally to generalized functions. The Fourier transform on $(E)_\beta^*$ resembles the finite dimensional Fourier transform in the sense that it is the adjoint of a continuous linear operator from the space of test functions into itself. This adjoint is a so called Fourier Gauss transform. In Gaussian Analysis, the Fourier-Gauss transform $\mathfrak{G}_{a,b}(\varphi)$ of $\varphi\in(E)_\beta$ is defined by
\begin{equation*}
\mathfrak{G}_{a,b}(\varphi)(y) = \int\limits_{E^*} \varphi(ax+by)\ d\mu(x), \quad a,b\in\C,\ 0\leq \beta < 1
\end{equation*}
see e.g.~\cite{Hida1}, \cite{Lee1} and \cite{Kuo2}. \\
The Fourier-Gauss transform is in $L((E)_\beta,(E)_\beta)$ and the operator symbol is given by
\begin{equation*}
 \widehat{\mathfrak{G}_{a,b}}(\xi,\eta)= \exp\left[\halb\left(a^2+b^2-1\right)\left\langle \xi,\xi\right\rangle+b\left\langle \xi,\eta\right\rangle\right], \quad \text{ for all } \xi,\eta\in E_\C,
\end{equation*}
see \cite[Theorem 11.29, p. 168-169]{Kuo2}.
Hence
\begin{equation}\label{FG-Trafo_expo}
\mathfrak{G}_{a,b}= \Gamma(b\ \mathrm{Id})\circ \exp(\halb(a^2+b^2-1)\Delta_G),
\end{equation}
where $\mathrm{Id}$ denotes the identity operator. Considering this formula, we show that we can find a suitable representation of the white noise space, such that equation \eqref{FG-Trafo_expo} will be just consisting of second quantization operators and is thus comparable to the approach in \cite{Wbock1} \\
This will be done by introducing generalized Wick tensors.
For $m\in\N$, $\kappa_{0,m}\in(E_\C^{\otimes m})^*_{sym}$ we define generalized Wick tensors by
\begin{equation*}
:x^{\otimes n}:_{\kappa_{0,m}}= \sum\limits_{k=0}^{\left\lfloor \frac{n}{m}\right\rfloor} \frac{n!}{(n-mk)!k!} (-\halb)^kx^{\otimes (n-mk)}\ \widehat{\otimes}\ \kappa_{0,m}^{\widehat{\otimes} k}.
\end{equation*}
We show that for each test function $\varphi$ in the space $(E)_\beta$, with $\frac{max(0,m-2)}{m}\leq \beta<1$, we have a unique representation
\begin{equation*}
\varphi=\sum\limits_{n=0}^{\infty} \left\langle :x^{\otimes n}:_{\kappa_{0,m}},\varphi_n\right\rangle,
\end{equation*}
with $\varphi_n\in (E_\C^{\otimes n})_{sym}$ for all $n\in\NZ$.\\
Indeed we can rewrite the Fourier-Gauss transform $\mathfrak{G}_{a,b}$ as a suitable second quantization operator $\left\{\Gamma_{\kappa_{0,r}}(bId)\right\}$. 
Using generalized Wick tensors it is very natural to find a larger class of one parameter transformation groups, see \ref{sec_main_result}. As an application we deduce explicitly the regular one parameter groups corresponding to the infinitesimal generators $a\Delta_G+bN$. \\

In order to classify regularity we use and extend the result from \cite[Theorem 5.6, p. 671]{Obata4}. There it is shown, that there exists a continuous homomorphism $C$ from $(E)_\beta^*$ to $L((E)_\beta,(E)_\beta)$. We show furthermore that this homomorphism is even an isomorphism from $(E)_\beta^*$ onto $\mathrm{Im}(C)$. We use this theorem to prove easily the differentiability of some one parameter transformation groups. Moreover we want to show that these transformation groups have equicontinuous generators. In particular we show that every differentiable one-parameter subgroup of $GL(X)$, where $X$ is a {nuclear} Fréchet space, has an equicontinuous generator, see \ref{oneparam_mainresult}.
%
%
%
%
\section{Preliminaries on white noise distribution theory}
\noindent

Starting point of the white noise distribution theory is the Gel'fand triple 
$$ E \subset L^2(\R,dt) \subset E^*,$$
where $E$ is a nuclear real countably Hilbert space which is densely embedded in the Hilbert space of square integrable functions w.r.t~ Lebesgue measure $L^2(\R,dt)$ and $E^*$ its topological dual space.\\
Via the Bochner-Minlos theorem, see e.g.~\cite{BK95}, we obtain the Gaussian measure $\mu$ on $E^*$ by its Fourier transform
$$\int_{E^*} \exp(i\langle x , \xi \rangle ) \, d\mu(x) = \exp(-\halb |\xi |_0^2), \quad \xi \in E 
$$
where $|.|_0$ denotes the Hibertian norm on $L^2(\R,dt)$. The topology on $E$ is induced by a positive self-adjoint operator $A$ on the space of real-valued functions $H = L^2(\R,dt)$ with $inf\ \sigma (A) > 1$ and Hilbert-Schmidt inverse $A^{-1}$ . We set $\rho := \left\|A ^{-1}\right\|_{OP}$ and $\delta := \left\|A ^{-1}\right\|_{HS}$. Note that the complexification $E_{\C}$ are equipped with the norms $\left|\xi\right|_p$ := $\left|A^p\xi\right|_0$ for $p \in \R$. We denote $H_\C :=  L^2(\R,\C,dt)$ furthermore  $E_{\C,p} := \left\{\xi\in E_\C^*|\ \left|\xi\right|_p <\infty\right\}$ and  $E^*_p := \left\{\xi\in E^*|\ \left|\xi\right|_p <\infty\right\}$, for $p \in \R$ resp.\\
Now we consider the following Gel'fand triple of White Noise test and generalized functions.
\begin{equation*}
(E)_\beta \subset (L^2) := L^2(E^*,\mu) \subset (E)^*_\beta,\ 0\leq \beta < 1 
\end{equation*} 
By the Wiener-Itô chaos decomposition theorem, see e.g.~\cite{Hida1, Obata1, Kuo2} we have the following unitary isomorphism between $(L^2)$ and the Boson Fock space $\Gamma(H_{\C})$: 
\begin{equation}
(L^2) \ni \Phi(x)\ = \sum\limits_{n=0}^\infty \left\langle :x^{\otimes n}:, f_n\right\rangle \ \leftrightarrow \ (f_n) \sim \Phi \in \Gamma(H_{\C}),\ f_n\in H_\C^{\hat\otimes n},
\end{equation} 
where $:x^{\otimes n}:$ denotes the Wick ordering of $x^{\otimes n}$ and $H_{\C}^{\hat\otimes n}$ is the symmetric tensor product of order $n$ of the complexification $H_{\C} \  of \ H$. Moreover the $(L^2)$ - norm of $\Phi \in (L^2)$ is given by
\begin{equation*}
\left|\Phi\right|_0^2 = \sum\limits_{n=0}^\infty n! \left|f_n\right|_0^2
\end{equation*} 
The elements in $(E)_\beta$ are called white noise test functions, the elements in $(E)^*_\beta$ are called generalized white noise functions. We denote by $\LL.,.\RR$ the canonical $\C$-bilinear form on $(E)_\beta^* \times (E)_\beta$. For each $\Phi \in (E)^*_\beta$ there exists a unique sequence $\left( F_n \right)_{n=0}^\infty, F_n \in (E_{\C}^{\hat{\otimes} n})^*$ such that 
\begin{equation}\label{def_duality}
\LL \Phi, \varphi \RR = \sum\limits_{n=0}^\infty n! \left\langle F_n,f_n\right\rangle,\indent   (f_n) \sim \varphi \in (E)_\beta 
\end{equation} 
Thus we have, see e.g.~\cite{Hida1, Obata1, Kuo2}: 
\begin{equation*}
(E)_\beta \ni \Phi \sim (f_n), 
\end{equation*} 
if and only if for all $ p \in \R  \text{ we have } \left|\Phi\right|_{p,\beta} := \left(\sum\limits_{n=0}^\infty (n!)^{1+\beta} \left|f_n\right|_p^2\right)^\halb < \infty$.
Moreover for its dual space we obtain 
\begin{equation}
(E)_\beta^* \ni \Phi \sim (F_n), 
\end{equation} 
if and only if there exists a $ p \in \R  \text{ such that } \left|\Phi\right|_{p,-\beta}:= \left(\sum\limits_{n=0}^\infty (n!)^{1-\beta} \left|F_n\right|_p^2\right)^\halb < \infty$.\\

For $p \in \R$ we define 
\begin{equation*}
(E)_{p,\beta}:= \left\{\varphi \in (E)_\beta^*:\ \left|\varphi\right|_{p,\beta} < \infty\right\}\ and\ (E)_{p,-\beta}:= \left\{\varphi \in (E)_\beta^*:\ \left|\varphi\right|_{p,-\beta} < \infty\right\}.
\end{equation*}
It follows
\begin{equation*}
(E)_\beta:= \projlim\limits_{p\rightarrow\infty} (E)_{p,\beta}
\end{equation*}
and 
\begin{equation*}
(E)_{\beta}^* = \indlim\limits_{p\rightarrow\ -\infty} (E)_{p,-\beta}
\end{equation*}
Moreover $(E)_\beta$ is a nuclear (F)-space.
In the following we use the abbreviation $(E):=(E)_0.$\\
\noindent The exponential vector or Wick ordered exponential is defined by 
\begin{equation}
\Phi_\xi(x) := \sum\limits_{n=0}^\infty \frac{1}{n!} \left\langle :x^{\otimes n}:, {\xi}^{\otimes n}\right\rangle \indent \ for \  \xi \in E_{\C}\ and\ x \in E^*.
\end{equation}  
For $y \in E^*_\C$ we use the same notation, which is only symbolic, and define $\Phi_y \in (E)_\beta^*$ by:
\begin{equation*}
(E)_\beta \ni \psi \sim (f_n)_{n\in\NZ}:\,  \LL \Phi_y , \psi\RR := \sum\limits_{n=0}^{\infty} \left\langle y^{\otimes n},f_n\right\rangle
\end{equation*}
Since $\Phi_\xi \in (E)_{\beta}$, for $\xi \in E_{\C}$ and $0\leq\beta<1$, we can define the so called $S$-transform of $\Psi\in (E)^*_{\beta}$ by
$$ S(\Psi)(\xi) =\LL \Phi_\xi, \Psi \RR,$$ 
Moreover we call $S(\Psi)(0)$ the generalized expectation of $\Psi \in (E)*_{\beta}$.\\
The Wick product of $\Theta \in(E)^*_{\beta} $ and $\Psi\in (E)^*_{\beta}$ is defined by $$\Psi\diamond\Theta := S^{-1}(S(\Psi)\cdot S(\Theta)) \in (E)^*_{\beta},$$ see e.g.~\cite{Hida1,Kuo2,Obata1}.\\

In the following we list some basic facts about integral kernel operators. 
\begin{definition}
Let $l,m$ be nonnegative integers. For each $\kappa_{l,m} \in (E_{\C}^ {\otimes (l+m)})^*$ the integral kernel operator $\Xi_{l,m}(\kappa_{l,m}) \in L((E)_\beta,(E)_\beta^*)$ with kernel distribution $\kappa_{l,m}$ (see \cite[Proposition 4.3.3, p. 82]{Obata1}) is defined by 
\begin{equation}
\Xi_{l,m}(\kappa_{l,m})(\varphi) = \sum\limits_{n=0}^{\infty} \frac{(n+m)!}{n!} \left\langle :x^{\otimes (l+n)}:,\ \kappa_{l,m} {\otimes}_m f_{n+m}\right\rangle, \indent \varphi \sim (f_n) \in (E)_\beta,
\end{equation} 
where $\kappa_{l,m} {\otimes}_m f_{n+m}$ is the right contraction (see \cite[Section 3.4, p. 53 ff]{Obata1}). 
\end{definition}
\begin{remark}

Note that $\left\langle :x^{\otimes (l+n)}:,\ \kappa_{l,m} {\otimes}_m f_{n+m}\right\rangle=\left\langle :x^{\otimes (l+n)}:,\ \kappa_{l,m} {\hat\otimes}_m f_{n+m}\right\rangle$, because $:x^{\otimes (l+n)}:$ is symmetric and ${\hat\otimes}_m$ means the right symmetric contraction.\\
The correspondence $\kappa_{l,m} \leftrightarrow \Xi_{l,m}(\kappa_{l,m})$ is one-to-one if $\kappa_{l,m} \in (E_{\C}^ {\otimes (l+m)})_{sym(l,m)}^*$.
Furthermore it is known (see \cite[Section 10.2, p. 123 ff]{Kuo2}):
\begin{equation}
\Xi_{l,m}(\kappa_{l,m}) \in L((E)_\beta,(E)_\beta) \Leftrightarrow \kappa_{l,m} \in (E_{\C}^ {\otimes l}) \otimes (E_{\C}^ {\otimes m})^*
\end{equation}
\end{remark}
By a dual argument to \cite[Theorem 4.3.9]{Obata1} it is known that an integral kernel operator $\Xi_{l,m}(\kappa_{l,m})$ is extendable to on operator in $L((E)_\beta^*,(E)_\beta^*)$ if and only if $\kappa_{l,m} \in (E_\C^{\otimes l})^* \otimes (E_\C^{\otimes m})$.
\\
The operator symbol of $\Xi\in L((E)_\beta,(E)_\beta^*)$ is defined as $\hat{\Xi} (\xi , \eta )$ := $\LL \Xi (\Phi_\xi), \Phi_\eta\RR$ for $\xi ,\ \eta \in E_\C$.

%
\section{Convergence of generalized functions}
Let $X$ be a nuclear (F)-space and $X^*$ its dual space w.r.t.~the strong dual topology. Recall that $L(X,X)$ is equipped with the {topology of bounded convergence}, namely the locally convex topology is defined by the semi-norms:
\begin{equation} \label{def_boundedconv}
\left\|T\right\|_{M,p} := \sup\limits_{\xi \in M} \left|T\xi\right|_p,\ T \in L(X,X)
\end{equation}
where M runs over the bounded subsets of X and $\left|\cdot\right|_p$ is an arbitrary semi-norm on X. The topology on $L(X^*,X^*)$ is defined in the same way. We state the following proposition:
%
\begin{proposition}\label{dual_contin}
Let X be a nuclear (F)-space and $X^*$ it's dual space with the strong dual topology. Then the mapping $^*:\ L(X,X)\rightarrow L(X^*,X^*)$, $T\longmapsto T^*$ is continuous. Furthermore if $X$ is a dense subspace of $X^*$, the mapping $^*$ is a topological isomorphism.
\end{proposition}
\begin{proof}
Let $M^*$ be a bounded subset of $E^*$ and $B$ be a bounded subset of $X$. Let $T\in L(X,X)$. Since $X$ is reflexive, we have
\begin{equation}\label{adjoint_iso}
\sup\limits_{\xi^*\in M^*} \left|T^*(\xi^*)\right|_{p_{B^\circ}} = \sup\limits_{\xi^*\in M^*,b\in B} \left|\left\langle T^*\xi^*,b\right\rangle\right| = \sup\limits_{b\in B} \left|Tb\right|_{p_{((M^*)^\circ)}}
\end{equation}
Now let $X$ be dense in $X^*$. Let $T\in L(X,X)$ and $T^*=0$. Then, for all $\xi\in X, \eta\in X$, we have $\left\langle T^*\xi,\eta\right\rangle=0$  and thus $\left\langle \xi,T\eta\right\rangle=0$.\\
Since  $X$ in $X^*$ we obtain $T=0$. Now let $Q\in L(X^*,X^*)$, then $(Q^*)^*=S$, which implies that $^*$ is surjective. The topological isomorphy follows by (\ref{adjoint_iso}).
\end{proof}

The following theorem can be found in e.g.~\cite{KLWS96+},\cite[Theorem 4.41, p. 127 f]{Hida1} and \cite[Theorem 8.6, p. 86f]{Kuo2}.
\begin{theorem}\label{generalconv}
Let $0\leq \beta < 1$. For all $n \in \N$ let $\Psi_n \in (E)_\beta^*$ and $F_n = S(\Psi_n)$, where $S$ denotes the $S$-transform. Then $(\Psi_n)_{n\in \N}$ converges strongly in $(E)_\beta^*$ if and only if the following conditions are satisfied:
\begin{enumerate}[(i)]
\item
$\lim\limits_{n\rightarrow \infty} F_n(\xi)$ exists for each $\xi \in E_\C$  
\item
There exist nonnegative constants $C, K, p \geq 0$, independent of n, such that\\
\begin{equation}
\left|F_n(\xi)\right|\leq C\exp(K{\left|\xi\right|}_p^{\frac{2}{1-\beta}}),\ \ \forall n \in \N,\ \xi \in E_\C.
\end{equation} 
\end{enumerate}
\end{theorem}
As an example we consider the convergence of Wick exponentials.
\begin{lemma}\label{wickconv}
Let ${\phi} \in (E)^*$ and ${\phi}\sim(\phi_0,\phi_1,\cdots,\phi_m,0,0,0,\cdots)$ be its Wiener-Itô chaos decomposition. Then the expression
\begin{equation}
\exp^\diamond({\phi}) := \sum\limits_{n=0}^\infty \frac{1}{n!} {\phi}^{\diamond n}
\end{equation}
converges in $(E)_\beta^*$, where $1>\beta\geq\frac{max(0,m-2)}{m}$.
\end{lemma}
\begin{proof}
We show that the conditions from Theorem \ref{generalconv} are fullfilled. Without loss of generality let $\phi_m \neq 0$.\\
Since ${\phi} \in (E)_\beta^*$ there exists $p > 0$ such that $\left|{\phi}\right|_{-p,\beta} < \infty$.
For $n\in\NZ$ and $\xi \in E_\C$ let $\Psi_n := \sum\limits_{k=0}^n \frac{1}{k!} {\phi}^{\diamond k}$ and $F_n(\xi) := \sum\limits_{k=0}^n \frac{1}{k!} \LL {\phi}^{\diamond k},\Phi_\xi\RR$.\\
Condition (i) from Theorem\ref{generalconv} is fulfilled since $F_n(\xi) = \sum\limits_{k=0}^n \frac{1}{k!} (S({\phi})(\xi))^k \rightarrow \exp(S({\phi})(\xi))$.\\
It follows $\left|F_n(\xi)\right| \leq  \exp(\left|S({\phi})(\xi)\right|)$.\\
Because for $r\geq min(2,m)$ we have
\begin{align*}
\left|S({\phi})(\xi)\right| &\leq\ \left|\phi_0\right|+ \cdots + \left| \langle\phi_m,\xi^{\otimes m}\rangle\right|\\
&\leq \left|\phi_0\right|+ \max\limits_{1\leq k\leq m}(\left|\phi_k\right|_{-p} )\cdot(m+\left|\xi\right|_{p}^m)\\
&\leq \left|\phi_0\right|+ \max\limits_{1\leq k\leq m}(\left|\phi_k\right|_{-p} )\cdot(m+1+\left|\xi\right|_{p}^r),
\end{align*}
 it follows that condition (ii) from Theorem \ref{generalconv} is fullfilled with $r = \frac{2}{1-\beta}$.
\end{proof}
The following is an immediate consequence of Lemma \ref{wickconv} and comparisation of the $S$-transforms.
\begin{corollary}\label{wickmult}
Let ${\phi},{\psi} \in (E)^*$ with corresponding chaos decompositions ${\phi} = (\phi_0,\phi_1,\cdots,\phi_m,0,0,0,...)$ and ${\psi} = (\psi_0,\psi_1,\cdots,\psi_m,0,0,0,...)$. Then the formula
\begin{equation*}
\exp^\diamond({\phi} +\psi) = \exp^\diamond({\phi}) \diamond \exp^\diamond({\psi})  
\end{equation*}
\end{corollary}
is valid in $(E)_\beta^*$, for all $1>\beta\geq\frac{1}{m} \max(0,m-2)$.
\section{Wick multiplication and convolution operator}
%
The Wick multiplication operator and its dual, often denoted as convolution operator are well known objects in white noise theory, see e.g.~\cite{Obata4}.
We give a proof, that $(E)_\beta^*$ with the Wick product may be considered as a commutative sub-algebra of $L((E_\beta),(E_\beta))$. For a similar approach see  also \cite{Obata4}.\\

%
%
\begin{proposition}$\;$
\begin{enumerate}[(i)]
\item
Let $\varphi\in (E)^*_\beta$. The Wick multiplication operator $M_\varphi$, defined by $$M_\varphi(\psi) := \varphi\diamond \psi$$ is a well defined operator in $L((E)_\beta,(E)^*_\beta$).
\item
Let $\varphi\in (E)^*_\beta$. The Wick multiplication operator $\widetilde{M_\varphi}$, defined by $$\widetilde{M_\varphi}(\psi) := \varphi\diamond \psi$$ is a well defined continuous extension of $M_\varphi$ to an operator in $L((E)^*_\beta,(E)^*_\beta$).
\end{enumerate}
\end{proposition}
\begin{proof}
For (ii), see \cite[Theorem 8.12., p. 92]{Kuo2}.\\
To prove (i) we know by \cite[Theorem 8.12., see Remark, p.92]{Kuo2}, that for any $p\geq 0$ there exist suitable $\alpha>0,\ c>0$, such that
\begin{align*}
\left|\varphi \diamond \psi\right|_{-p-\alpha,-\beta}\ &\leq\ c\cdot\left|\varphi\right|_{-p,-\beta}\cdot \left|\psi\right|_{-p,-\beta}\\
                                             &\leq\ c\cdot\left|\varphi\right|_{-p,-\beta}\cdot \left|\psi\right|_{p,\beta}
\end{align*}
Thus $\varphi \diamond (\cdot) \in L((E)_\beta,(E)_\beta^*)$.
\end{proof}
\begin{definition}\label{def_conv_op}
The dual operator of the Wick operator $\widetilde{M_\varphi}$ is called \textbf{convolution operator} and denoted by $C_\varphi$.
\end{definition}
The next proposition gives the Fock expansion of the Wick multiplication operator.
\begin{proposition}\label{wick_sum_op}
Let $\varphi\in(E)_\beta^*$ and $\varphi\sim (\varphi_1,\varphi_2,\cdots)$. Then $M_\varphi$, considered as $M_\varphi\in L((E)_\beta,(E)_\beta^*$), has the following Fock expansion:
\begin{equation*}
M_\varphi = \sum\limits_{n\in\NZ} \Xi_{n,0}(\varphi_n).
\end{equation*}
\end{proposition}
\begin{proof}
Let $\xi, \eta\in E_\C$. Then
\begin{align*}
\LL M_\varphi \Phi_\xi,\Phi_\eta \RR &= S(\varphi\diamond \Phi_\xi)(\eta) = S(\varphi)(\eta)\cdot S(\Phi_\xi)(\eta)\\
& = S(\varphi)(\eta)\cdot e^{\left\langle \xi,\eta\right\rangle}\\
\end{align*}
By \cite[(10.21) p. 145]{Kuo2} the Taylor expansion of $M_\varphi$ is given by
\begin{equation*}
S(\varphi)(\eta)=\sum\limits_{n\in\NZ} \left\langle \varphi_n,\eta^{\otimes n}\otimes \xi^0\right\rangle
\end{equation*}
completes the proof, compare also \cite[Proposition 4.5.3, p.98-99]{Obata1}.
\end{proof}
\begin{proposition}\label{conv_sum_op}
Let $\varphi\in (E)_\beta^*$ and $\varphi\sim (\varphi_1,\varphi_2,\cdots)$. The convolution operator $C_\varphi$ is in $L((E)_\beta,(E)_\beta)$ and has the Fock expansion
\begin{equation*}
C_\varphi = \sum\limits_{n\in\NZ} \Xi_{0,n}(\varphi_n).
\end{equation*}
\end{proposition}
\begin{proof}
By definition we have $C_\varphi=\widetilde{M_\varphi}^* \in L((E)_\beta,(E)_\beta)$. Moreover ${M_\varphi}^*\big|(E)_\beta = C_\varphi$, since $\widetilde{M_\varphi}$ is an extension of $M_\varphi$.\\
Then, by Proposition \ref{wick_sum_op} we have
\begin{equation*}
C_\varphi={M_\varphi}^* = \left(\sum\limits_{n\in\NZ} \Xi_{n,0}(\varphi_n)\right)^*=\sum\limits_{n\in\NZ} \Xi_{0,n}(\varphi_n).
\end{equation*} 
\end{proof}
The following statement is an immediate consequence from the definition of integral kernel operators.
\begin{corollary}\label{eigenstate}
Let $\varphi\in (E)_\beta^*$ and $\varphi\sim (\varphi_1,\varphi_2,\cdots)$. Then
\begin{equation*}
\ C_\varphi(\Phi_\xi)= \left[S(\varphi)(\xi)\right]\Phi_\xi,\quad \xi\in E_\C:
\end{equation*}
\end{corollary}

\begin{lemma}\label{expect2}
Let $\Xi \in L((E)_\beta,(E)_\beta^*)$ and $\Xi = \sum\limits_{l,m=0}^{\infty} \Xi_{l,m}(\kappa_{l,m})$ be the corresponding unique representation as sum of integral kernel operators. Then:
\begin{equation}
\Xi(\Phi_0) = \sum\limits_{l=0}^{\infty} \left\langle :x^{\otimes l}:, \kappa_{l,0}\right\rangle
\end{equation} 
\end{lemma}
\begin{proof}
Consider the equation
\begin{equation}
\Xi_{l,m}(\kappa_{l,m})(\varphi) = \sum\limits_{n=0}^{\infty} \frac{(n+m)!}{n!} \left\langle :x^{\otimes (l+n)}:, \kappa_{l,m} {\hat\otimes}_m f_{n+m}\right\rangle, \quad \varphi \sim (f_n) \in (E)_\beta,
\end{equation} 
then for $\Phi_0 \sim (f_n)$ with $f_0 = 1$ and $f_n = 0$ for all $n \geq 1$:
\begin{equation*}
\Xi(\Phi_0) =\sum\limits_{l=0}^{\infty} \left\langle :x^{\otimes l}:, \kappa_{l,0}\ \hat\otimes_0\  1\right\rangle
\end{equation*}
\end{proof}

\begin{theorem}\label{subalgebra1}
Let $0\leq \beta<1$.
The mapping
$$
C: ((E)_\beta^*,+,.,\diamond) \ to (L((E)_\beta,(E)_\beta),+,.,\circ)\quad
              \varphi \mapsto\ \mathscr{C}_{\varphi} 
$$
is an injective continous vector algebra homomorphism.
In particular $C$ is a topological isomorphism from $(E)_\beta^*$ to $\mathrm{Im}(C)$.\\
\end{theorem}
\begin{proof}
Linearity is given by the definition of integral kernel operators. In order to prove homomorphy, let $\xi\in E_\C,\ \varphi_1,\varphi_2\in(E)_\beta^*$. Then by Corollary \ref{eigenstate}
\begin{align*}
C_{\varphi_1} \circ C_{\varphi_2} (\Phi_\xi) &= S(\varphi_1)(\xi)\cdot S(\varphi_2)(\xi) \cdot \Phi_\xi\\
&= S(\varphi_1\diamond\varphi_2)(\xi) \cdot \Phi_\xi= C_{\varphi_1\diamond \varphi_2} (\Phi_\xi)
\end{align*}
For injectivity let $\varphi \in (E)_\beta^*$ and $\mathscr{C}_\varphi = 0$. Then  ${\mathscr{C}_\varphi}^* = 0$ and by \ref{conv_sum_op} and \ref{expect2} we obtain $\varphi = {\mathscr{C}_\varphi}^*(\Phi_0)=0$.\\
To prove continuity let $\varphi \in (E)_\beta^*$ and $\varphi:=\sum\limits_{n\in\NZ} \left\langle :x^{\otimes n}:,\varphi_n\right\rangle$ be the corresponding chaos decomposition.
Now let $p \geq 0$. Moreover let $r>0$, such that $$(2^{\frac{1-\beta}{^2}}\rho^{r}<1),$$ where $\rho := \left\|A ^{-1}\right\|_{OP}$. Choose $q>0$ such that $\rho^{-q}\geq 2$. Then by \cite[Theorem 10.5, p. 128]{Kuo2} we have for all $\phi \in (E)_\beta$ and $m \in \NZ$
\begin{align*}
\left|\Xi_{0,m}(\varphi_m)(\phi)\right|_{p,\beta}
&\leq\ (m!2^m)^{\frac{1-\beta}{2}}\cdot \left|\varphi_m\right|_{-(p+q+r)}\cdot \left|\phi\right|_{(p+q+r),\beta}\\
&\leq\ m!^{\frac{1-\beta}{2}}\left|\varphi_m\right|_{-(p+q)}  (2^{\frac{1-\beta}{2}}\rho^{r})^m \cdot \left|\phi\right|_{(p+q+r),\beta}\\
\end{align*}
Now let $K:=(\sum\limits_{m\in\NZ}(2^{\frac{1-\beta}{2}}\rho^{r})^{2m})^\halb$.
Then by Cauchy-Schwartz we obtain
\begin{align*}
\sum\limits_{m\in\NZ}\left|\Xi_{0,m}(\varphi_m)(\phi)\right|_{p,\beta}
&\leq\ \sum\limits_{m\in\NZ}(m!^{\frac{1-\beta}{2}}\left|\varphi_m\right|_{-(p+q)}  (2^{\frac{1-\beta}{2}}\rho^{r})^m\cdot \left|\phi\right|_{(p+q+r),\beta}\\
 &\leq\ K\left|\varphi\right|_{-(p+q),\beta} \cdot \left|\phi\right|_{(p+q+r),\beta}\\
\end{align*}
Now let $M \subset (E)_\beta$ be bounded. Then $\sup\limits_{\phi\in M}\left|\phi\right|_{p+q+r,\beta} \leq \mathrm{const}(M,p,q,r) < \infty$ and finally $\left\|\mathscr{C}_{\varphi}\right\|_{M,\left\{p,\beta\right\}}\leq K\mathrm{const}(M,p,q,r) \cdot \left|\varphi\right|_{-(p+q),\beta}$\\
To prove the last claim, we use Proposition \ref{dual_contin} and and represent $C^{-1}$ as composition of the two continuous mappings:
\begin{equation*}
C^{-1}: \mathscr{C}_{\varphi} \stackrel{(.)^*}{\longmapsto} {\mathscr{C}_\varphi}^* \mathop{\longmapsto}\limits^\text{point}_\text{evaluation}  {\mathscr{C}_\varphi}^*(\Phi_0) = \varphi.
\end{equation*} 
\end{proof}
The following statement is a consequence from Lemma \ref{wickconv} and Theorem \ref{subalgebra1}.
\begin{corollary}\label{expo}
Let $\varphi \in (E)^*$ and $\varphi\sim(\varphi_0,\varphi_1,\cdots,\varphi_m,0,0,0,\cdots)$ be it's chaos decomposition. Then the expression
\begin{equation}
\exp(\mathscr{C}_{\varphi}) := \sum\limits_{n=0}^\infty \frac{1}{n!} (\mathscr{C}_{\varphi})^n
\end{equation}
converges in $L((E)_\beta,(E)_\beta)$, where $1>\beta\geq \max(0,\frac{m-2}{m})$.
\end{corollary}

\section{Regular one parameter groups}
\noindent
Let $X$ a nuclear (F)-space  over $\C$ and $(\left|\cdot\right|_n)_{n\in\N}$ be a family of Hilbertian semi-norms, with $\left|\cdot\right|_{n} \leq \left|\cdot\right|_{n+1}$ for all $n\in\N$, topologizing $X$,  see e.g.~\cite[Proposition 1.2.2 and proposition 1.3.2]{Obata1}. For $x^*\in (X,\left|\cdot\right|_n)^*$ we define
\begin{equation*}
\left|x^*\right|_{-n} := \sup\limits_{\left|x\right|_n \leq 1} \left|\left\langle x^*,x\right\rangle\right|
\end{equation*}
By a Hahn-Banach argument, 
we have for all $x \in X:$
\begin{equation*}
\left|x\right|_n = \sup \left\{\left|\left\langle x^*,x\right\rangle\right|\ : x^*\in (X,\left|\cdot\right|_n)^* \wedge \left|x^*\right|_{-n} \leq 1 \right\}
\end{equation*}
for all semi-norms $\left|\cdot\right|_n,\ \ n\in\N$.
\begin{definition}
A family $\left\{\Omega_\theta\right\}_{\theta\in\R} \subset L(X)$ is called differentiable in $\theta_0\in\R$, if 
\begin{equation*}
\lim\limits_{\theta\rightarrow \theta_0} \frac{\Omega_{\theta}\phi-\Omega_{\theta_0}\phi}{\theta-\theta_0}
\end{equation*}
converges in $X$ for any $\phi\in X$. In that case a linear operator $\Omega'_{\theta_0}$ from $X$ into itself is defined by
\begin{equation*}
\Omega'_{\theta_0}\phi:=\lim\limits_{\theta\rightarrow \theta_0} \frac{\Omega_{\theta}\phi-\Omega_{\theta_0}\phi}{\theta-\theta_0}
\end{equation*}
$\left\{\Omega_\theta\right\}_{\theta\in\R}$ is called differentiable, if it is differentiable at each $\theta\in\R$.\\
In the following we use the abbreviation:\\
\begin{equation*}
\Omega' :=\Omega'_{0}
\end{equation*}
\end{definition}
\begin{proposition}\label{compactconv}
Let $\theta_0 \in\R$ and $\left\{\Omega_\theta\right\}_{\theta\in\R} \subset L(X)$ be a family of operators which is differentiable in $\theta_0$. Then $\Omega'_{\theta_0}$ is continuous, i.e. $\Omega'_{\theta_0}\in L(X)$. Moreover we have uniformly convergence on every compact (or equivalently, bounded) subset of $X$, i.e.:
\begin{equation*}
\lim\limits_{\theta\rightarrow \theta_0} \sup\limits_{\phi\in K} \left|\frac{\Omega_\theta\phi-\Omega_{\theta_0}\phi}{\theta-\theta_0} - \Omega'_{\theta_0}\phi\right|_n = 0
\end{equation*}
for any $n\in\N$ and any compact (or bounded) subset $K\subset X$.
\end{proposition}
\begin{proof}
First note that a subset of the nuclear space $X$ is compact if and only if it is closed and bounded.
The assertion follows by an application of the Banach-Steinhaus theorem.\\
(For the uniform convergence on compact subsets see e.g.~\cite[4.6 Theorem, p. 86]{Schaefer1}. )
\end{proof}
\begin{definition}
A family $\left\{\Omega_\theta\right\}_{\theta\in\R} \subset L(X)$ is called regularily differentiable in $\theta_0\in\R$, if there exists $\Omega '_{\theta_0}\in L(X)$ such that for any $n\in\N$ there exists $m\in\N$ such that
\begin{equation*}
\lim\limits_{\theta\rightarrow \theta_0}\sup\limits_{\left|\phi\right|_m\leq 1}\left|\frac{\Omega_{\theta}\phi-\Omega_{\theta_0}\phi}{\theta-\theta_0}-\Omega'_{\theta_0}\phi\right|_n  = 0.
\end{equation*}
$\left\{\Omega_\theta\right\}_{\theta\in\R}$ is called regular, if it is regularily differentiable at each $\theta\in\R$.
\end{definition}
\begin{definition}
Let $\left\{\Omega_\theta\right\}_{\theta\in\R}$ be a family of operators in $L(X)$ with
\begin{equation*}
\forall \theta_1,\ \theta_2\in\R:\ \Omega_{\theta_1+\theta_2} = \Omega_{\theta_1}\circ\Omega_{\theta_2},\ \ \ \Omega_0=\mathrm{Id}
\end{equation*}
Then, as easily seen, $\left\{\Omega_\theta\right\}_{\theta\in\R}$ is a subgroup of $GL(X)$ and is called a one-parameter subgroup of $GL(X)$.
\end{definition}
In the following we collect facts about one-parameter subgroups, for details see e.g.~\cite[Section 5.2]{Obata1}. 
%
\begin{lemma}\label{oneparam_remark}
Let $\left\{\Omega_\theta\right\}_{\theta\in\R}$ be a one-parameter subgroup of $GL(X)$.
\begin{enumerate}[(i)]
\item Then $\left\{\Omega_\theta\right\}_{\theta\in\R}$ is differentiable at each $\theta\in\R$ if and only if $\left\{\Omega_\theta\right\}_{\theta\in\R}$ is differentiable at 0, i.e. there exists $\Omega'\in L(X)$ such that
\begin{equation*}
\lim\limits_{\theta\rightarrow 0} \left|\frac{\Omega_{\theta}\phi-\phi}{\theta}-\Omega'\phi\right|_n = 0.
\end{equation*}
for all $\phi\in X$ and $n\in\N$.
\item Let $\left\{\Omega_\theta\right\}_{\theta\in\R}$ be a differentiable subgroup of $GL(X)$. Then we have
\begin{enumerate}[a)]
\item
$\Omega'$ is an element of $L(X)$, further unique and is called the infinitesimal generator of the differentiable one parameter subgroup $\left\{\Omega_\theta\right\}_{\theta\in\R}$ of $GL(X)$. Conversely a differentiable one parameter subgroup $\left\{\Omega_\theta\right\}_{\theta\in\R}$ of $GL(X)$ is uniquely defined by it's infinitesimal generator, see \cite[Proposition 5.2.2, p. 119]{Obata1}. If $\left\{\Omega_\theta\right\}_{\theta\in\R}$ is regular the infinitesimal generator $\Omega'$ is called an equicontinuous generator.
\item
$\left\{\Omega_\theta\right\}_{\theta\in\R}$ is infinitely many differentiable at each $\theta\in\R$ and
\begin{equation*}
\forall n\in\N:\ \frac{d^n}{d\theta^n}  \Omega_\theta = (\Omega')^n\circ\Omega_\theta = \Omega_\theta\circ (\Omega')^n
\end{equation*}
\item
$\R\rightarrow L(X),\ \theta\mapsto \Omega_\theta$ is continuous.\\
Note that $L(X)$ is equipped with the topology of bounded convergence, compare \cite[Section 5.2, Eq. (5.27), p. 119]{Obata1}.
\end{enumerate}
\end{enumerate}
\end{lemma}
\begin{lemma}\label{oneparam_equic}
Let $T: \R\rightarrow L(X)$ be a continuous mapping. If $K$ is a compact subset of $\R$, then $T(K)$ is equicontinuous.
\end{lemma}
\begin{proof}
$T(K)$ is compact, hence bounded, especially pointwisely bounded, by the definition of the topology of bounded convergence. The theorem of Banach and Steinhaus then completes the proof.
\end{proof}
\begin{proposition}\label{oneparam_compact}\indent
Let $\left\{\Omega_\theta\right\}_{\theta\in\R}$  be a differentiable one parameter subgroup of $GL(X)$.
\begin{enumerate}[(i)]
\item
$\left\{\Omega_\theta\right\}_{\theta\in\R}$ is {regular} if and only if for any $n\in\N$ there exists $m\in\N$ such that
\begin{equation*}
\lim\limits_{\theta\rightarrow 0}\sup\limits_{\left|\phi\right|_m\leq 1}\left|\frac{\Omega_{\theta}\phi-\phi}{\theta}-\Omega'\phi\right|_n  = 0.
\end{equation*}
\item
For any compact(or bounded) subset $K\subset X$ we have
\begin{equation*}
\lim\limits_{\theta\rightarrow 0}\sup\limits_{\phi\in K}\left|\frac{\Omega_{\theta}\phi-\phi}{\theta}-\Omega'\phi\right|_n  = 0.
\end{equation*}
\item
Let $\theta_0 > 0$. Then for any $n\in\N$ there exists $m\in\N$ and $K(n,m)>0$, such that for all $\phi\in X$
\begin{equation*}
\sup\limits_{\left|\theta\right|\leq \theta_0,\theta\neq 0}\left|\frac{\Omega_\theta\phi-\phi}{\theta}\right|_n \leq K(n,m)\left|\phi\right|_m
\end{equation*}
\end{enumerate}
\end{proposition}
\begin{proof}
The statement (i) is easily verified, moreover (ii) is an immediate consequence of Proposition \ref{compactconv}.
To prove statement (iii) we have by (ii) and Proposition \ref{oneparam_remark} (ii) a), that the mapping
$$\R\longrightarrow L(X), \quad 
\theta\longmapsto 
\left\{ 
\begin{array}{ll} 
\frac{\Omega_\theta-\mathrm{Id}}{\theta}&,\ \theta \neq 0\\ 
\Omega'&,\ \theta = 0
\end{array}
\right.$$
is continuous. Then the claim is obtained by Lemma \ref{oneparam_equic}.
\end{proof}
%
\begin{theorem}[Regularity]\label{oneparam_mainresult}
Let $\left\{\Omega_\theta\right\}_{\theta\in\R}$  be a differentiable one-parameter subgroup of $GL(X)$. Then $\left\{\Omega_\theta\right\}_{\theta\in\R}$ is regular.
\end{theorem}
\begin{proof}
Let $n\in\N$ and $\xi\in X$, $\eta \in (X,\left|\cdot\right|_n)^*$. We define for all $t \in\R:$ 
\begin{align*}
f(t):= \left\langle \eta,\Omega_t\xi\right\rangle
\end{align*}
Then we have by Proposition \ref{oneparam_remark}
\begin{align*}
f'(t)  &= \left\langle \eta,\Omega'\Omega_t\xi\right\rangle\\
f''(t) &= \left\langle \eta,(\Omega')^2\Omega_t\xi\right\rangle
\end{align*}
Now let $\theta_0>0$ be fixed. Then by Proposition \ref{oneparam_remark} (ii) c) and Lemma \ref{oneparam_equic}, $\left\{(\Omega')^2\Omega_t\right\}_{\left|t\right|\leq\theta_0}$ is equicontinuous. \\
Thus there exists $m\in\NZ,\ K=K(\theta_0,n,m,\Omega')>0$ such that
\begin{equation*}
\max\limits_{\left|t\right|\leq\theta_0} \left|f''(t)\right|\leq K\left|\xi\right|_{n+m}\left|\eta\right|_{-n}
\end{equation*}
Let $\theta\in\R$ with $\left|\theta\right|\leq\theta_0$.
By Taylor expansion we have
\begin{align*}
\left|f(\theta)-f(0)-\theta\cdot f'(0)\right|
&\leq \frac{\left|\theta\right|^2}{2}  \max\limits_{\left|t\right|\leq\theta_0} \left|f''(t)\right|\\
\leq &\frac{\left|\theta\right|^2}{2} K\left|\xi\right|_{n+m}\left|\eta\right|_{-n}
\end{align*}
and for $\theta\neq 0$
\begin{align*}
\sup\limits_{\left|\xi\right|_{n+m}\leq 1} \sup\limits_{\left|\eta\right|_{-n}
\leq 1} \left|\frac{\left\langle \eta,\Omega_\theta\xi\right\rangle-\left\langle \eta,\xi\right\rangle}{\theta}-\left\langle \eta,\Omega'\xi\right\rangle\right|
\leq &\frac{\left|\theta\right|}{2} K.
\end{align*}
Hence
\begin{equation*}
\sup\limits_{\left|\xi\right|_{n+m}\leq 1} \left|\frac{\Omega_\theta\xi-\xi}{\theta}-\Omega'\xi\right|_n
\leq \frac{\left|\theta\right|}{2} K
\end{equation*}
such that 
\begin{equation*}
\lim\limits_{\theta\rightarrow 0}\sup\limits_{\left|\xi\right|_{n+m}\leq 1} \left|\frac{\Omega_\theta\xi-\xi}{\theta}-\Omega'\xi\right|_n = 0.
\end{equation*}
\end{proof}
%
%
%
%
\begin{proposition}\label{basics1_group} 
Let $\left\{S_\theta\right\}_{\theta\in\R}$ be a differentiable one parameter subgroup of $GL(X)$ and $\left\{T_\theta\right\}_{\theta\in\R}$ be a family of operators in $L(X)$ which is differentiable in zero with $\lim\limits_{\theta\rightarrow 0} T_\theta\phi = \phi$ for all $\phi\in X$. Then
\begin{equation*}
\frac{d}{d\theta}\big|_{\theta=0} S_\theta \circ T_\theta = S' + T'
\end{equation*}
If $\left\{T_\theta\right\}_{\theta\in\R}$ is regularily differentiable in zero, then\\
$\left\{S_\theta\circ T_\theta\right\}_{\theta\in\R}$ is also regularily differentiable in zero.
\end{proposition}
\begin{proof}\indent
Let $n\in\N,\ \phi\in X$. Then $\left\{S_\theta|\ \left|\theta\right|\leq 1\right\}$ is equicontinuous by Proposition \ref{oneparam_remark}, (ii) c). Then there exist a semi-norm $\left|\cdot\right|_m$, with $m\in\N$ and $m\geq n$, and a constant $K>0$ such that $\left|S_\theta(\phi)\right|_n \leq K\cdot \left|\phi\right|_m$ for all $\phi\in X$ and $\theta\in\R$ with $\left|\theta\right|\leq 1$. Now let $\phi\in X$ be arbitrarily chosen. For  $\theta\in\R$ with $\left|\theta\right|\leq 1$ we have:
\begin{multline*}
\left|\frac{S_\theta\circ T_\theta(\phi)-\phi}{\theta}-S'\phi-T'\phi\right|_n
=\left|\frac{S_\theta\circ (T_\theta(\phi)-\phi)}{\theta} + \frac{S_\theta(\phi)-\phi}{\theta}-S'\phi-T'\phi\right|_n\\
\leq\left|\frac{S_\theta\circ (T_\theta(\phi)-\phi)}{\theta}-T'\phi\right|_n + \left|\frac{S_\theta(\phi)-\phi}{\theta}-S'\phi\right|_n\\
\leq K\cdot\left|\frac{(T_\theta(\phi)-\phi)}{\theta}-S_{-\theta}T'\phi\right|_m + \left|\frac{S_\theta(\phi)-\phi}{\theta}-S'\phi\right|_n\\
\leq K\cdot\left|\frac{(T_\theta(\phi)-\phi)}{\theta}-T'\phi\right|_m + 
K\cdot\left|S_{-\theta}(T'\phi) - (T'\phi)\right|_m + 
\left|\frac{S_\theta(\phi)-\phi}{\theta}-S'\phi\right|_n,
\end{multline*}
Where last inequalilty follows by Theorem \ref{oneparam_mainresult} and the continuity of $T'$.
\end{proof}
We use the following notation, due to \cite[Eq. (4.66), p. 106]{Obata1}
\begin{definition}
$$\left\{ 
\begin{array}{ll} 
\gamma_n(T) &:= \sum\limits_{k=0}^{n-1} \mathrm{Id}^{\otimes k}\otimes T \otimes \mathrm{Id}^{\otimes (n-1-k)},\ n\geq 1\\ 
\gamma_0(T) &:= 0
\end{array}
\right.
$$
\end{definition}

Now let $T\in L(E_\C)$. We recall the definition of  the second quantization operator  of $T$, denoted by  $\Gamma(T)$ and $d\Gamma(T)$, the differential second quantization operator. Suppose $\phi\in (E)_\beta$ is given as
\begin{equation*}
\phi(x) = \sum\limits_{n=0}^{\infty} \left\langle :x^{\otimes n}:,f_n\right\rangle,\ \ x\in E^*,\ f_n\in (E^{\otimes n}_\C)_{sym}
\end{equation*}
as usual. We put
\begin{equation*}
\Gamma(T)\phi(x) := \sum\limits_{n=0}^{\infty} \left\langle :x^{\otimes n}:,T^{\otimes n}f_n\right\rangle 
\end{equation*}
and 
\begin{equation*}
d\Gamma(T)\phi(x) := \sum\limits_{n=0}^{\infty} \left\langle :x^{\otimes n}:,\gamma_n(T)f_n\right\rangle 
\end{equation*}
We have $\Gamma(T),\ d\Gamma(T)\in L((E)_\beta)$, for all $0\leq \beta < 1$.

\begin{lemma}\label{tensor_of_transf}
Let $\left\{\Omega_\theta\right\}_{\theta\in\R}$ be a differentiable one parameter subgroup of $GL(X)$ with infinitesimal generaor $\Omega'$. Then, for all $n\in\N$, $\left\{\Omega_\theta^{\otimes n}\right\}_{\theta\in\R}$ is a differentiable one parameter subgroup of $GL(X^{\otimes n})$ with
\begin{equation*}
\frac{d}{d\theta}\big|_{\theta=0}\ \Omega_\theta^{\otimes n} = \gamma_n(\Omega')
\end{equation*}
\end{lemma}
\begin{proof}
We show the case $n=2$. The general case is similar.\\
It holds
\begin{equation*}
\Omega_\theta \otimes \Omega_\theta = (\Omega_\theta \otimes \mathrm{Id}) \circ (\mathrm{Id} \otimes \Omega_\theta)
\end{equation*}
Now apply Proposition \ref{basics1_group}
\end{proof}
\begin{proposition}\label{basics2_group}\indent
Let $0\leq\beta<1$. Further let $\left\{\Omega_\theta\right\}_{\theta\in\R}$ be a differentiable one parameter subgroup of $GL(E_\C)$ with infinitesimal generator $\Omega'$. Then $\Gamma(\Omega_\theta)_{\theta\in\R}$ is a differentiable one-parameter subgroup of $GL((E)_\beta)$ with infinitesimal generator $d\Gamma(\Omega')$.
\end{proposition}
\begin{proof}
Note that for $0\leq \beta<1$, the sequence $(\alpha(n))_{n\in\NZ}$ with $\alpha(n):= n!^\beta$ fullfills the conditions of \cite[Theorem 4.2, p.696]{Chung3}, where a detailed proof, based on the characterization theorem, is given. On the other hand the expected result is easily seen by Lemma \ref{tensor_of_transf}.
\end{proof}
For a similar statement like Proposition \ref{basics2_group}, see \cite[5.4.5, p. 130-131]{Obata1}.

\begin{proposition}\label{basics4_group}\indent
Let $\left\{\Omega_\theta\right\}_{\theta\in\R}$ be a differentiable one parameter subgroup of $GL(E_\C)$. Further let $r\in\N$, $1>\beta\geq\frac{max(0,r-2)}{r}$ and $\kappa_{0,r}\in (E_\C^{\otimes r})^*_{sym}$.\\\\
Then $\left\{\exp(\Xi_{0,r}((\Omega_\theta^{\otimes r})^*{\kappa_{0,r}}))\right\}_{\theta\in\R}\subset GL((E)_\beta)$ is an in zero differentiable family of operators with
\begin{equation*}
\frac{d}{d\theta}\big|_{\theta=0} \exp(\Xi_{0,r}((\Omega_\theta^{\otimes r})^*{\kappa_{0,r}})) = \Xi_{0,r}((\gamma_r(\Omega'))^*{\kappa_{0,r}})\circ \exp(\Xi_{0,r}(({\kappa_{0,r}})).
\end{equation*}
\end{proposition}
\begin{proof}
First, because $\R$ is a metric space, it is enough to consider sequential convergence, i.e.~the limit process for any arbitrary sequence $(\theta_n)_{n\in\N}$ in $\R$ with $\lim\limits_{n\rightarrow \infty} \theta_n = 0$. Consider the sequence $\left\{\exp^\diamond\left\langle :x^{\otimes r}:,(\Omega_{\theta_n}^{\otimes r})^*\kappa_{0,r}\right\rangle\right\}_{n\in\N}$. Then the limit process will be transferred to $L((E)_\beta,((E)_\beta))$ by continuity, using Theorem \ref{subalgebra1}.\\
Now let $(\theta_n)_{n\in\N}$ be an arbitrary sequence in $\R$ with $\lim\limits_{n\rightarrow \infty} \theta_n = 0$ and $\theta_n\neq 0$ for all $n\in\N$. Further define
\begin{equation*}
\varphi_{n}:=\frac{\exp^\diamond (\left\langle :x^{\otimes r}:,(\Omega^{\otimes r}_{\theta_n})^*\kappa_{0,r}\right\rangle) - \exp^\diamond (\left\langle :x^{\otimes r}:,\kappa_{0,r}\right\rangle)}{\theta_n}
\end{equation*}
for all $n\in\N$.
Then, by Lemma \ref{wickconv}, we have $\varphi_{n}\in (E)_\beta^*$ for all $n\in\N$. We verify the conditions of Theorem \ref{generalconv}.\\
Let $\xi\in E_\C$. Then for all $n\in\N$
\begin{equation*}
S(\varphi_{n})(\xi) = \frac{\exp(\left\langle \kappa_{0,r},(\Omega^{\otimes r}_{\theta_n})\xi^{\otimes r}\right\rangle) - \exp(\left\langle \kappa_{0,r},\xi^{\otimes r}\right\rangle)}{\theta_n}
\end{equation*}
First note that for all $n\in\N$ the $S$-transform $S(\varphi_{n})$ is entire holomorphic. By Lemma \ref{tensor_of_transf} and Proposition \ref{oneparam_remark} we obtain $\left\{\Omega_\theta^{\otimes r}\right\}_{\theta\in\R}$ as a regular one-parameter subgroup of $(E_\C^{\otimes r})_{sym}$ with
\begin{equation*}
\frac{d}{d\theta} \Omega^{\otimes r}_{\theta} = \gamma_r(\Omega')\Omega_\theta^{\otimes r}.
\end{equation*}
Hence the function  $\theta\mapsto \left\langle \kappa_{0,r},(\Omega^{\otimes r}_{\theta})\xi^{\otimes r}\right\rangle$ is infinitely often differentiable on $\R$ and the same holds for $\theta\mapsto\exp(\left\langle \kappa_{0,r},(\Omega^{\otimes r}_{\theta})\xi^{\otimes r}\right\rangle)$ as composition of two infinitely many differentiable functions with
\begin{equation*}
\frac{d}{d\theta} \exp(\left\langle \kappa_{0,r},(\Omega^{\otimes r}_{\theta})\xi^{\otimes r}\right\rangle) = \left\langle \kappa_{0,r},\gamma_r(\Omega')\Omega_\theta^{\otimes r}\xi^{\otimes r}\right\rangle\cdot \exp(\left\langle \kappa_{0,r},(\Omega^{\otimes r}_{\theta})\xi^{\otimes r}\right\rangle)
\end{equation*}
Hence $\lim\limits_{n\rightarrow\infty}(S(\varphi_{n})(\xi))$ exists and we have:
\begin{equation*}
\lim\limits_{n \rightarrow \infty} S(\varphi_n)(\xi) = \left\langle \kappa_{0,r},\gamma_r(\Omega')\xi^{\otimes r}\right\rangle\cdot \exp(\left\langle \kappa_{0,r},\xi^{\otimes r}\right\rangle)
\end{equation*}
by the chain rule and Proposition \ref{basics2_group}\\
For the growth estimate let $p\geq 0$.
Without loss of generality let $\left|\theta_n\right|<1$ for all $n\in\N$. Then, since $\left\{\Omega_\theta\right\}_{\left|\theta\right|\leq 1}$ and $\left\{\Omega' \Omega_\theta\right\}_{\left|\theta\right|\leq 1}$ are compact by \ref{oneparam_equic}, there exists a $q\geq 0$ such that, $\forall \theta\in\R,\ \left|\theta\right|\leq 1$ we have $\left|\Omega_{\theta}(\xi)\right|_p\leq \left|\xi\right|_{p+q}$ and $\left|\Omega'\Omega_{\theta}(\xi)\right|_p\leq \left|\xi\right|_{p+q}$. Then, by the mean value theorem and by Proposition \ref{oneparam_remark} (ii) b), it follows for each $n\in\N$:
\begin{align*}
\left|S(\varphi_{n})(\xi)\right|\ &\leq \sup\limits_{\left|\theta\right|\leq 1} \left|\left\langle \kappa_{0,r},\gamma_r(\Omega')\Omega_\theta^{\otimes r}\xi^{\otimes r}\right\rangle\cdot \exp(\left\langle \kappa_{0,r},(\Omega^{\otimes r}_{\theta})\xi^{\otimes r})\right\rangle\right|\\
&\leq \left|\kappa_{0,r}\right|_{-p}\cdot r\left|\xi\right|^r_{p+q}\cdot \exp(\left|\kappa_{0,r}\right|_{-p}\cdot \left|\xi\right|^r_{p+q}))\\
&\leq \left|\kappa_{0,r}\right|_{-p}\cdot r(r!)\cdot \exp((1+\left|\kappa_{0,r}\right|_{-p})\cdot \left|\xi\right|^r_{p+q}))\\
&\leq \left|\kappa_{0,r}\right|_{-p}\cdot r(r!)\cdot \exp((1+\left|\kappa_{0,r}\right|_{-p})\cdot (1+\left|\xi\right|^{\frac{2}{1-\beta}}_{p+q})),\\
\end{align*}
for all $1>\beta\geq \frac{max(0,r-2)}{r}$.
The claim is a consequence of Theorem \ref{subalgebra1}.
\end{proof}
%
%
The same idea as in the proof of Proposition \ref{basics4_group} combined with Theorem \ref{oneparam_mainresult} leads to the following result.
\begin{proposition}\label{regular}
Let $r\in\N$, $1>\beta\geq\frac{1}{r}\max(0,r-2)$. Further let $\varphi \in (E)_\beta^*$ with $\varphi\sim (\varphi_0,\cdots,\varphi_r,0,0,0,\cdots)$.
Define $\Xi_\theta := \exp(\theta\ \mathscr{C}_{\varphi})$ for $\theta \in \R$. Then $\left\{\Xi_\theta : \theta\in\R\right\}$ is a regular one parameter subgroup of $GL((E)_\beta)$ with infinitesimal generator $\mathscr{C}_{\varphi}$.
\end{proposition}
\begin{example}
Let $y \in E^*$. As simple example consider the operator $\Xi_{0,1}(y)$. Recall that $D_y = \Xi_{0,1}(y)$ and the $translation\ operator$ $T_y$ = $\exp(D_y)$.\\
Let $z \in \C$. From $z \Delta_G$ = $\Xi_{0,2}(z\tau)$, where $\tau$ is the trace operator, we conclude that $\exp(z\Delta_G) \in L((E)_\beta,(E)_\beta)$, for all $0\leq\beta<1$.
\end{example}
%

%
 \section{Generalized Wick tensors and an application}
%
Our goal in this section is to rewrite Fourier-Gauss transforms as second quantization operators. This will be accomplished by suitable basis-transformations which will be explicitly calculated. For this purpose we introduce generalized Wick tensors. In this context Fourier-Gauss transforms appear as second quantization operators of the form 
$\left\{\Gamma_{\kappa_{0,r}}(\Omega_\theta)\right\}$. As an application we deduce explicitly the regular one parameter groups corresponding to the infinitesimal generators $a\Delta_G+bN$.\\
First we repeat some definitions.
Let $a,b\in\C,\ 0\leq \beta < 1$. The Fourier-Gauss transform $\mathfrak{G}_{a,b}(\varphi)$ of $\varphi\in(E)_\beta$ is defined to be the function
\begin{equation*}
\mathfrak{G}_{a,b}(\varphi)(y) = \int\limits_{E^*} \varphi(ax+by)\ d\mu(x)
\end{equation*}
The Fourier-Gauss transform is in $L((E)_\beta,(E)_\beta)$ and the operator symbol is given by
\begin{equation*}
\widehat{\mathfrak{G}_{a,b}}(\xi,\eta)= exp\left[\halb\left(a^2+b^2-1\right)\left\langle \xi,\xi\right\rangle+b\left\langle \xi,\eta\right\rangle\right],\quad \text{for all } \xi,\eta\in E_\C
\end{equation*}
see e.g.~\cite[Theorem 11.29, p. 168-169]{Kuo2}.
Hence
\begin{equation*}
\mathfrak{G}_{a,b}= \Gamma(b\ \mathrm{Id})\circ exp(\halb(a^2+b^2-1)\Delta_G)
\end{equation*}
By \cite[Lemma 11.22, p. 163]{Kuo2} the operator symbol of the Fourier transform is given by
\begin{equation*}
\widehat{\mathfrak{F}}(\xi,\eta)= exp(-i\left\langle \xi,\eta\right\rangle-\halb\left\langle \eta,\eta\right\rangle), \quad \text{for all } \xi,\eta\in E_\C.
\end{equation*}
Consequently
\begin{equation*}
\mathfrak{F}= exp(-\halb\Delta_G)^*\circ \Gamma(-i\ \mathrm{Id})
\end{equation*}
Moreover the operator symbol of the Fourier-Mehler transform is given by
\begin{equation*}
\widehat{\mathfrak{F_\theta}}(\xi,\eta)= exp(e^{i\theta}\left\langle \xi,\eta\right\rangle+ \frac{i}{2}e^{i\theta}sin\theta\left\langle \eta,\eta\right\rangle),\quad \text{for all } \xi,\eta\in E_\C,\ \theta\in\R,
\end{equation*}
see e.g.~\cite[11., p. 180]{Kuo2}.
Thus the Fourier-Mehler transform is given by the formula
\begin{equation*}
\mathfrak{F}_\theta= \left[exp(\frac{i}{2}e^{i\theta}sin\theta\ \Delta_G)\right]^*\circ \Gamma(e^{i\theta}\ \mathrm{Id}).
\end{equation*}
In the following let $\mathscr{G}_\theta$ denote the adjoint of the Fourier-Mehler transform $\mathfrak{F}_\theta$.\\
Finally by \cite[Proposition 4.6.9, p. 105]{Obata1} the operator symbol of the scaling operator is given by
\begin{equation*}
\widehat{S_\lambda}(\xi,\eta)= exp\left((\lambda^2-1)\left\langle \xi,\xi\right\rangle/2 + \lambda\left\langle  \xi,\eta\right\rangle\right),\quad \text{for all } \xi,\eta\in E_\C,\ \lambda\in\C.
\end{equation*}
and consequently
\begin{equation*}
S_\lambda = \Gamma(\lambda\ \mathrm{Id})\circ \exp\left(\frac{\lambda^2-1}{2}\ \Delta_G\right)
\end{equation*}
\begin{definition}
Let $m\in\N$ and $\kappa_{0,m} \in ((E_\C)^{\otimes m})^*_{sym}$. For $x\in E^*$ we define the renormalized tensor power $:x^{\otimes n}:_{\kappa_{0,m}}$ as follows:
\begin{equation*}
:x^{\otimes n}:_{\kappa_{0,m}} = \sum\limits_{k=0}^{\left\lfloor \frac{n}{m}\right\rfloor} \frac{n!}{(n-mk)!k!}\cdot \left(-\halb\right)^k\cdot x^{\otimes (n-mk)}\ \widehat{\otimes}\ \left(\kappa_{0,m}\right)^{\widehat{\otimes} k}
\end{equation*}
\end{definition}
For the usual tensor power $x^{\otimes n},\ x\in E^*$ it holds the following relation:, see \cite[Corollary 2.2.4, p. 25]{Obata1}
\begin{equation*}
x^{\otimes n} = \sum\limits_{k=0}^{\left\lfloor \frac{n}{2}\right\rfloor} \frac{n!}{(n-2k)!k!}\cdot \left(\halb\right)^k\cdot :x^{\otimes (n-2k)}:_{\tau}\ \widehat{\otimes}\ \left(\tau\right)^{\widehat{\otimes} k}
\end{equation*}
As a simple example we have for $x\in E$ the relation $:x^{\otimes n}:\ =\ :x^{\otimes n}:_\tau$.\\
More usually is the abbreviation $:x^{\otimes n}:_{\sigma^2}\ \stackrel{def}{=}\ :x^{\otimes n}:_{\sigma^2\tau}$.\\
Note that $\kappa_{0,m} = 0$ is permitted, e.g. $:x^{\otimes n}:_{0} = x^{\otimes n}$. But we don't permit $m=0$.\\\\\\
\noindent
There exists $\Theta$ in $L((E),(E))$ defined by $\Theta(\exp(\left\langle .,\xi\right\rangle)) := \Phi_\xi $, see e.g.~ \cite[Theorem 6.2]{Kuo2}. Since $\int\limits_{E^*} \Phi_\xi\ d\mu = 1$ we call $\Theta$ the renormalization operator.\\

\begin{proposition} \label{normal} 
\begin{enumerate}[(i)]
\item
$\Theta = \exp(-\halb \Delta_G)$
\item
For all $ f_m \in (E_\C^{\otimes m})_{sym}$ we have: $$ \Theta(\left\langle x^{\otimes m}, f_m\right\rangle) = \left\langle :x^{\otimes m}:, f_m\right\rangle$$ 
\end{enumerate}
\end{proposition}
\begin{proof}
By Corollary \ref{eigenstate} we have for all $\xi \in E_\C$: \\
\begin{align*}
\exp(\halb\Delta_G)\Phi_\xi &=\ exp(\halb\left\langle \xi,\xi\right\rangle)\Phi_\xi\\
 &=\ \exp(\halb\left\langle \xi,\xi\right\rangle)\exp(-\halb\left\langle \xi,\xi\right\rangle)e^{\left\langle .,\xi\right\rangle}\\
&=\ e^{\left\langle .,\xi\right\rangle} 
\end{align*}
By Proposition \ref{regular} $\Theta$ is invertible and the first statement is proved. To proof the second, let $m \in \NZ$. Note that for $m < 2n$ we have 
\begin{equation*}
\Xi_{0,2n}(\tau^{\otimes n}) (\left\langle :x^{\otimes m}:, f_m\right\rangle) = 0
\end{equation*}
Thus $\Theta^{-1} = \exp(\halb\Delta_G) = \sum\limits_{n=0}^{\infty} \frac{1}{2^n n!}\Xi_{0,2n}(\tau^{\otimes n})$.\\
Let $m \geq 2n\ and\ \delta_{i,j}$ be the Kronecker symbol. Then by \cite[Proposition 4.3.3, Eq. (4.23), p. 82]{Obata1} it holds:
\begin{align*}
\Xi_{0,2n}(\tau^{\otimes n}) (\left\langle :x^{\otimes m}:, f_m\right\rangle) &=\ \sum\limits_{k=0}^{\infty} \frac{(k+2n)!}{k!} (\left\langle :x^{\otimes k}:,\tau^{\otimes n} \otimes_{2n} \delta_{k+2n,m}\cdot f_{m}\right\rangle)\\
&=\ \frac{m!}{(m-2n)!} \left\langle :x^{\otimes m-2n}:, \tau^{\otimes n} \otimes_{2n} f_m\right\rangle\\
&=\ \frac{m!}{(m-2n)!} \left\langle :x^{\otimes m-2n}:\ \otimes\ \tau^{\otimes n}, f_m\right\rangle\\
&=\ \frac{m!}{(m-2n)!} \left\langle :x^{\otimes m-2n}:\ \hat \otimes\ \tau^{\hat\otimes n}, f_m\right\rangle\\
\end{align*}
where the last equation is due to the symmetricity of $f_m$.\\
Then
\begin{align*}
\exp(\halb\Delta_G)(\left\langle :x^{\otimes m}:, f_m\right\rangle)  &=\ \left\langle \sum\limits_{n=0}^{\left\lfloor{\frac{m}{2}}\right\rfloor} \frac{m!}{(m-2n)!n!2^n} :x^{\otimes m-2n}:\ \hat\otimes\ \tau^{\hat\otimes n}, f_m\right\rangle \\
&=\ \left\langle x^{\otimes m},f_m\right\rangle,
\end{align*}
compare also\cite[Corollary 2.2.4]{Obata1}.
\end{proof}
Note that $e^{\left\langle \cdot,\xi\right\rangle} \in (E)$ since $\Phi_\xi \in (E)$.
\begin{corollary}
Let $\xi \in E_\C$. Then the series $e^{\left\langle \cdot,\xi\right\rangle} = \sum\limits_{n=0}^{\infty} \frac{1}{n!}\left\langle .,\xi\right\rangle^n$ converges in $(E)$.
\end{corollary}
\begin{proof}
Let $\xi \in E_\C$. Since $$\Phi_\xi = \sum\limits_{n=0}^{\infty} \frac{1}{n!} \left\langle :x^{\otimes n}:, \xi^{\otimes n}\right\rangle$$ converges in $(E)$ also $$\exp(\halb\Delta_G)\Phi_\xi = \sum\limits_{n=0}^{\infty} \frac{1}{n!} \left\langle x^{\otimes n}, \xi^{\otimes n}\right\rangle$$ converges in $(E)$.
\end{proof}
%
%
The following proposition generalizes Proposition \ref{normal}.
\begin{proposition}\label{general_normal} 
Let $r\in\N$ and $\kappa_{0,r} \in ((E_\C)^{\otimes r})^*_{sym}$. For all $f_m \in (E_\C^{\otimes m})_{sym}$ we have:
\begin{equation*}
\exp(-\halb \Xi_{0,r}(\kappa_{0,r}))(\left\langle x^{\otimes m}, f_m\right\rangle) = \left\langle :x^{\otimes m}:_{\kappa_{0,r}}, f_m\right\rangle
\end{equation*}
\end{proposition}
\begin{proof}
Let $m \in \NZ$. Note that for $m < rn$ we have 
\begin{equation*}
\Xi_{0,rn}(\kappa_{0,r}^{\otimes n}) (\left\langle :x^{\otimes m}:, f_m\right\rangle) = 0
\end{equation*}
We have $\exp(-\halb \Xi_{0,r}(\kappa_{0,r})) = \sum\limits_{n=0}^{\infty} \frac{1}{n!}(-\halb)^n\Xi_{0,rn}(\kappa_{0,r}^{\otimes n})$.\\
Let $m \geq rn\text{ and }\delta_{i,j}$ be the Kronecker symbol. Then by Theorem \ref{subalgebra1} and Proposition \ref{normal}, using \cite[Proposition 4.3.3, Eq. (4.23), p. 82]{Obata1}, we have:
\begin{multline*}
\Xi_{0,rn}(\kappa_{0,r}^{\otimes n}) (\left\langle x^{\otimes m}, f_m\right\rangle)=\exp(\halb\Delta_G)\circ \Xi_{0,rn}(\kappa_{0,r}^{\otimes n}) (\left\langle :x^{\otimes m}:, f_m\right\rangle)\\
=\ \exp(\halb\Delta_G)\ \sum\limits_{k=0}^{\infty} \frac{(k+rn)!}{k!} (\left\langle :x^{\otimes k}:,\kappa_{0,r}^{\otimes n} \otimes_{rn} \delta_{k+rn,m}\cdot f_{m}\right\rangle)\\
=\ \frac{m!}{(m-rn)!} \left\langle x^{\otimes m-rn}, \kappa_{0,r}^{\otimes n} \otimes_{rn} f_m\right\rangle\\
=\ \frac{m!}{(m-rn)!} \left\langle x^{\otimes m-rn}\ \otimes\ \kappa_{0,r}^{\otimes n},\ f_m\right\rangle=\ \frac{m!}{(m-rn)!} \left\langle x^{\otimes m-rn}\ \hat\otimes\ \kappa_{0,r}^{\hat\otimes n}, f_m\right\rangle
\end{multline*}
where the last equation follows because $f_m$ is symmetric.\\

Then by the above definition we have:
\begin{align*}
\exp(-\halb \Xi_{0,r}(\kappa_{0,r}))(\left\langle x^{\otimes m}, f_m\right\rangle)  &=\ \left\langle \sum\limits_{n=0}^{\left\lfloor{\frac{m}{r}}\right\rfloor} \frac{m!}{(m-rn)!n!}(-\halb)^n x^{\otimes m-rn}\ \hat\otimes\ \kappa_{0,r}^{\hat\otimes n}, f_m\right\rangle \\
&=\ \left\langle :x^{\otimes m}:_{\kappa_{0,r}}, f_m\right\rangle
\end{align*}
\end{proof}
\begin{notation}
In the following we use $0\cdot\tau$ in order to express that we consider $0\in (E^{\otimes 2})^*$.
\end{notation}
\begin{theorem}[Representation theorem] \label{ren_representation}
Let $r\in\N$, $1>\beta\geq\frac{max(0,r-2)}{r}$ and $\kappa_{0,r} \in ((E_\C)^{\otimes r})^*_{sym}$. Then each $(\varphi_n)_{n\in\NZ}\sim\varphi\in (E)_\beta$ has a unique decomposition
\begin{equation*}
 \varphi(x) = \sum\limits_{n=0}^{\infty} \left\langle :x^{\otimes n}:_{\kappa_{0,r}},(\psi_n)_{\kappa_{0,r}}\right\rangle, quad \text{ for all }x \in E^*
\end{equation*}
with $(\psi_n)_{\kappa_{0,r}}\in (E_\C^{\otimes n})_{sym}$, which we denote as $\kappa_{0,r}$ - representation of $\varphi$.
\end{theorem}
\begin{proof}
First note that $\varphi$ has a unique Wiener-Itô chaos decomposition $(\varphi_n)_{n\in\NZ}$, see e.g.~\cite[Theorem 3.1.5]{Obata1}.
Then the existence and uniqueness of the above representation follows by the bijectivity of  $\exp(-\halb\Delta_G)\ \circ\ \exp(\halb \Xi_{0,r}(\kappa_{0,r}))$. Recall Corollary \ref{expo} and the following chain of mappings:
\begin{equation*} 
\left\langle :x^{\otimes n}:_{\kappa_{0,r}},\varphi_n\right\rangle\stackrel{\exp(\halb \Xi_{0,r}(\kappa_{0,r}))}{\longmapsto}\left\langle x^{\otimes n},\varphi_n\right\rangle\stackrel{\exp(-\halb\Delta_G)}{\longmapsto}\left\langle :x^{\otimes n}:,\varphi_n\right\rangle
\end{equation*}
\end{proof}
\begin{remark}
 The $S$-transform of $\varphi\in (E)$ is a restriction of the $0\tau$ - representation of $\exp(\halb\Delta_G)\varphi$ from $E^*$ to $E$.\\
Note that we do not claim that the above decomposition is an orthogonal decomposition with respect to the measure $\mu$, like the chaos decomposition. We claim only the uniqueness.
\end{remark}
\begin{definition}\label{def_ren_op}
Let $r\in\N$, $1>\beta\geq\frac{max(0,r-2)}{r}$, $\kappa_{0,r}\in (E_\C^{\otimes r})^*_{sym}$. For $T\in L((E)_\beta,(E)_\beta)$, we define:
\begin{equation*}
T_{\kappa_{0,r}} := \exp(\halb(\Delta_G-\Xi_{0,r}(\kappa_{0,r}))\circ T \circ \exp(-\halb(\Delta_G-\Xi_{0,r}(\kappa_{0,r})))
\end{equation*}
$T_{\kappa_{0,r}}$ is called the renormalization of $T$ corresponding to $\kappa_{0,r}$. Obviously  $T_{\kappa_{0,r}}\in L((E)_\beta,(E)_\beta)$.
For $\Omega\in L((E)_\beta)$ we abbreviate 
$$\Gamma_{\kappa_{0,r}}(\Omega):=(\Gamma(\Omega))_{\kappa_{0,r}}$$ resp. 
$$d\Gamma_{\kappa_{0,r}}(\Omega):=(d\Gamma(\Omega))_{\kappa_{0,r}}$$.\\
\end{definition}
\begin{remark}
It is clear, that $T_{\kappa_{0,r}}$ acts formally on white noise test functions in $\kappa_{0,r}$ - representation like $T$ on white noise test functions in the standard representation as Boson Fock space. We precise this statement by the following proposition.
\end{remark}
\begin{proposition}
Let $r\in\N$, $1>\beta\geq\frac{max(0,r-2)}{r}$, $\kappa_{0,r}\in (E_\C^{\otimes r})^*_{sym}$. For $T\in L((E)_\beta,(E)_\beta)$ and $\varphi \in (E)_\beta$ with $\varphi = \sum\limits_{n=0}^{\infty} \left\langle :x^{\otimes n}:,f_n\right\rangle$, we use the notation $T(\varphi)= \sum\limits_{n=0}^{\infty}\left\langle :x^{\otimes n}:,(f_n)_T\right\rangle$. Then it follows:
\begin{equation*}
T_{\kappa_{0,r}}(\sum\limits_{n=0}^{\infty} \left\langle :x^{\otimes n}:_{\kappa_{0,r}},f_n\right\rangle = \sum\limits_{n=0}^{\infty} \left\langle :x^{\otimes n}:_{\kappa_{0,r}},(f_n)_T\right\rangle
\end{equation*}
\end{proposition}
\begin{proof}
The expression $\sum\limits_{n=0}^{\infty} \left\langle :x^{\otimes n}:_{\kappa_{0,r}},f_n\right\rangle$ is well defined because 
$$\sum\limits_{n=0}^{\infty} \left\langle :x^{\otimes n}:_{\kappa_{0,r}},f_n\right\rangle= \exp(-\halb \Xi_{0,r}(\kappa_{0,r}))\ \circ\ \exp(\halb\Delta_G)(\sum\limits_{n=0}^{\infty} \left\langle :x^{\otimes n}:,f_n\right\rangle).$$
By Corollary \ref{wickmult} and Theorem \ref{subalgebra1} we have
\begin{equation*}
\exp(-\halb(\Delta_G-\Xi_{0,r}(\kappa_{0,r})))= \exp(-\halb\Delta_G)\ \circ\ \exp(\halb \Xi_{0,r}(\kappa_{0,r})).
\end{equation*}
Then using Proposition \ref{normal} and Proposition \ref{general_normal}, the claim follows with the same idea as in the proof of Theorem \ref{ren_representation}.
\end{proof}
%
%
The following formula is suitable for the calculation of renormalized second quantization operators. Recall that for all $T\in L(E_\C,E_\C)$, we have $\Gamma(T)\in L((E)_\beta,(E)_\beta)$, where $0\leq\beta<1$.
\begin{corollary}\label{ren_sec_quan}
Let $r\in\N$, $1>\beta\geq\frac{max(0,r-2)}{r}$, $\kappa_{0,r}\in (E_\C^{\otimes r})^*_{sym}$. Then for all $T\in L(E_\C,E_\C)$ we have
\begin{multline*}
\Gamma_{\kappa_{0,r}}(T)= \Gamma(T) \circ \exp(\halb\Xi_{0,2}((T^{\otimes 2}-\mathrm{Id}^{\otimes 2})^*(\tau)))\\
\circ
\exp(-\halb\Xi_{0,r}((T^{\otimes r}-\mathrm{Id}^{\otimes r})^*(\kappa_{0,r}))).
\end{multline*}
\end{corollary}
\begin{proof}
On the one hand we calculate
\begin{align*}
&\ \Gamma(T)\circ \exp(\halb\Xi_{0,2}((T^{\otimes 2}-\mathrm{Id}^{\otimes 2})^*(\tau)))\circ
\exp(-\halb\Xi_{0,r}((T^{\otimes r}-\mathrm{Id}^{\otimes r})^*(\kappa_{0,r})))\Phi_\xi\\
=\ &\ (\exp(\halb\left\langle \tau,(T^{\otimes 2}-1)\xi^{\otimes 2}\right\rangle))\cdot(\exp(-\halb\left\langle (\kappa_{0,r}),(T^{\otimes r}-1)\xi^{\otimes r}\right\rangle))\Phi_{T\xi}
\end{align*}
On the other hand by Definition \ref{def_ren_op}
\begin{align*}
& \Gamma_{\kappa_{0,r}}(T)(\Phi_\xi)\\
&=\exp(\halb(\Delta_G-\Xi_{0,r}(\kappa_{0,r})))\circ \Gamma(T) \circ \exp(-\halb(\Delta_G-\Xi_{0,r}(\kappa_{0,r})))(\Phi_\xi)\\
&=\exp(-\halb(\left\langle \xi,\xi>-<\kappa_{0,r},\xi^{\otimes r}\right\rangle))\cdot \exp(\halb(\Delta_G-\Xi_{0,r}(\kappa_{0,r})))(\Phi_{T\xi})\\
&=\exp(-\halb(\left\langle \xi,\xi>-<\kappa_{0,r},\xi^{\otimes r}\right\rangle))\cdot \exp(\halb(\left\langle T\xi,T\xi\right\rangle-\left\langle \kappa_{0,r},(T\xi)^{\otimes r}\right\rangle))\Phi_{T\xi}\\
&=(\exp(\halb\left\langle \tau,(T^{\otimes 2}-1)\xi^{\otimes 2}\right\rangle))\cdot(\exp(-\halb\left\langle (\kappa_{0,r}),(T^{\otimes r}-1)\xi^{\otimes r}\right\rangle))\Phi_{T\xi}
\end{align*}
\end{proof}
We consider some concrete examples to Corollary \ref{ren_sec_quan}.
\begin{example}\label{ex_sec_quant}
Let $b \in \C,\ \theta\in \R$ and $\kappa_{0,2}\in (E_C^{\otimes 2})_{sym}$. Then
\begin{enumerate}[(i)]
\item
\begin{equation*}
\Gamma_{\kappa_{0,2}}(b\ \mathrm{Id}) = \Gamma(b\ \mathrm{Id})\exp(\halb(b^2-1)\Xi_{0,2}(\tau-\kappa_{0,2}))
\end{equation*}
\item
\begin{equation*}
\Gamma_{\kappa_{0,2}}(e^{i\theta}\ \mathrm{Id})= \Gamma(e^{i\theta}\ \mathrm{Id})\circ  \exp\left(i\cdot e^{i\theta}sin\theta\ \Xi_{0,2}(\tau-\kappa_{0,2})\right)
\end{equation*}
\item
\begin{equation*}
\Gamma_{\halb\tau}(e^{i\theta}\ \mathrm{Id}) = \Gamma(e^{i\theta}\ \mathrm{Id})\circ \exp\left(\frac{i}{2}e^{i\theta}sin\theta\ \Delta_G\right)
\end{equation*}
By \cite[Lemma 5.6.1, p.140]{Obata1} we conclude that $\Gamma_{\halb\tau}(e^{i\theta}\ \mathrm{Id})$ is the adjoint operator of the Fourier-Mehler transform $\mathscr{F}_\theta$.
\end{enumerate}
\end{example}

The next example gives conditions under which the Fourier-Gauss transforms are renormalized second quantized operators.
\begin{example}\label{ex_gauss}
Let $a,b\in \C$ and $b\notin \left\{-1,1\right\}$. Choose $\sigma$ with $\sigma^2 = \frac{a^2}{1-b^2}$. Then
\begin{equation*}
\Gamma_{\sigma^2\tau}(b\ \mathrm{Id})= \mathfrak{G}_{a,b}
\end{equation*}
Which follows immediately by Example \ref{ex_sec_quant}(i) with $\kappa_{0,2}=\sigma^2\tau$.
\end{example}

\begin{remark}
In the special case $a=0$, we have for the scaling operator $S_b := \mathfrak{G}_{0,b}=\Gamma_{0\tau}(b\ \mathrm{Id})$, i.e. $S_b(\left\langle x^{\otimes n},f_n\right\rangle) = \left\langle x^{\otimes n},b^n \cdot f_n\right\rangle$ for all $n\in\NZ,\ f_n \in (E_\C^{\otimes n})_{sym}$.\\
With $a,b\in \C$, $b\notin \left\{-1,1\right\}$ and $\sigma^2 = \frac{a^2}{1-b^2}$, we have
\begin{equation*}
\Gamma_{\sigma^2\kappa_{0,2}}(b\ \mathrm{Id}) = \Gamma(b\ \mathrm{Id})\circ \exp(\halb\left[a^2\Xi_{0,2}(\kappa_{0,2})+(b^2-1)\ \Delta_G\right])
\end{equation*}
\end{remark}
\begin{theorem}\label{sec_main_result}
Let $\left\{T_\theta\right\}_{\theta\in\R}$ be a differentiable one parameter subgroup\\
 of $GL(E_\C)$ with infinitesimal generator $T'$. Further let $r\in\N$, $1>\beta\geq\frac{max(0,r-2)}{r}$ and $\kappa_{0,r}\in (E_\C^{\otimes r})^*_{sym}$. Then $\left\{\Gamma_{\kappa_{0,r}}(T_\theta)\right\}_{\theta\in\R}$ is a regular one parameter subgroup of $GL((E)_\beta)$ with\\
\begin{equation*}
d\Gamma_{\kappa_{0,r}}(T')=\frac{d}{d\theta}\big|_{\theta=0} \Gamma_{\kappa_{0,r}}(T_\theta)= d\Gamma(T') +\halb\Xi_{0,2}(\gamma_2(T')^*\tau)-\halb\Xi_{0,r}(\gamma_r(T')^*\kappa_{0,r})
\end{equation*}
\end{theorem}
\begin{proof}
By Corollary \ref{ren_sec_quan} we have 
\begin{multline*}
\Gamma_{\kappa_{0,r}}(T_\theta)=\Gamma(T_\theta)\circ \exp(\halb\Xi_{0,2}((T_\theta^{\otimes 2}-\mathrm{Id}^{\otimes 2})^*(\tau)))\\
\circ\exp(-\halb\Xi_{0,r}((T_\theta^{\otimes r}-\mathrm{Id}^{\otimes r})^*(\kappa_{0,r}))).
\end{multline*}
First note that $(G\circ \Gamma(T_\theta)\circ G^{-1})_{\theta\in\R}$ is obviously a regular one-parameter subgroup of $GL((E)_\beta)$ for all $G\in GL((E)_\beta)$. \\
Consequently  $(\Gamma_{0\tau}(T_\theta))_{\theta\in\R}$, with $$\Gamma_{0\tau}(T_\theta)=\Gamma(T_\theta)\circ \exp(\halb\Xi_{0,2}((T_\theta^{\otimes 2}-\mathrm{Id}^{\otimes 2})^*(\tau))),$$ is a regular one-parameter subgroup of $GL((E)_\beta)$. Because
\begin{multline*}
\Gamma_{\kappa_{0,r}}(T_\theta)= \left[\Gamma(T_\theta)\circ \exp(\halb\Xi_{0,2}((T_\theta^{\otimes 2}-\mathrm{Id}^{\otimes 2})^*(\tau)))\right]\\
\circ\ 
\exp(-\halb\Xi_{0,r}((T_\theta^{\otimes r}-\mathrm{Id}^{\otimes r})^*(\kappa_{0,r}))),
\end{multline*}
the claim follows by Proposition \ref{basics1_group}, then Proposition \ref{basics2_group} and Proposition \ref{basics4_group}.
\end{proof}
\begin{corollary}\label{ren_infin_gen}
Let $a,b\in \C$ and $b\neq 0$.\\
Then $bN_{(1-\frac{a}{b})\tau}=\frac{d}{d\theta}\big|_{\theta=0}\Gamma_{(1-\frac{a}{b})\tau}(e^{b\theta}\ \mathrm{Id}) = a\Delta_G + bN$.
\end{corollary}
\begin{proof}
By \cite[Proposition 4.6.13., p.107]{Obata1} we have $N:=\Xi_{1,1}(\tau)=d\Gamma(\mathrm{Id})$. Because $bId$ is the infinitesimal generator of $\left\{e^{b\theta} \mathrm{Id}\right\}_{\theta\in\R}$, it follows by Proposition \ref{basics2_group} that $$\frac{d}{d\theta}\big|_{\theta=0}(\Gamma(e^{b\theta}\ \mathrm{Id}))=d\Gamma(b\ \mathrm{Id}) = \Xi_{1,1}(b\tau)= bN,$$ for $b\in\C$.\\
On the other hand, in the general case $r\in\N$, $1>\beta\geq\frac{max(0,r-2)}{r}$ and $\kappa_{0,r}\in (E_\C^{\otimes r})^*_{sym}$, it holds from Theorem \ref{sec_main_result}, with $T'=\mathrm{Id}$, that
\begin{equation}\label{eq_sec_number}
N_{\kappa_{0,r}}=\frac{d}{d\theta}\big|_{\theta=0}\Gamma_{\kappa_{0,r}}(e^{\theta}\mathrm{Id}) = N + \Delta_G - \frac{r}{2} \Xi_{0,r}({\kappa_{0,r}}).
\end{equation}
Multiplicating both sides with b, the claim is immediate with $\kappa_{0,r}=(1-\frac{a}{b})\tau$.
\end{proof}

\begin{remark}
Let $b\neq 0$. The regular one-parameter subgroup $$\left\{\Gamma_{(1-\frac{a}{b})\tau}(e^{b\theta}\ \mathrm{Id})\right\}_{\theta\in\R}\subset GL((E)_\beta),\ 0\leq\beta<1$$ can, Example \ref{ex_gauss}, be identified as the Fourier-Gauss transforms $\left\{\mathfrak{G}_{x,e^{b\theta}}\right\}$, where $x^2 = (1-\frac{a}{b})(1-e^{2b\theta})$.\\
The case $b=0$ is solved by \ref{regular}. We get $\left\{\mathfrak{G}_{\sqrt{2a\theta},1}\right\}$ as solution.\\
Note, that the special choice of x from $x^2$ has no influence, because $\left\{\mathfrak{G}_{x,e^{b\theta}}\right\}$ only depends on $x^2$ and $e^{b\theta}$.
\end{remark}
We summarize this discussion using the definition of the Fourier-Gauss transform in \cite[Definition 11.24, p. 164]{Kuo2}. The following theorem is a generalization of the Mehler formula for the Ornstein-Uhlenbeck semigroup, see e.g.~\cite[p. 237]{Hida1} 
\begin{theorem}
Let $a,b \in \C$, $0\leq\beta<1$. Then $a\cdot\Delta_G + b\cdot N$ is the infinitesimal generator of the following regular transformation group $\left\{P_{a,b,t}\right\}_{t\in\R}\subset GL((E)_\beta)$:
\begin{enumerate}[(i)]
	\item 
	if $b\neq 0$ then for all $\varphi\in (E)_\beta,\ t\in\R$:
	\begin{equation*}
	\ P_{a,b,t}(\varphi)=\int\limits_{E^*}\varphi(\sqrt{(1-\frac{a}{b})(1-e^{2bt})}\cdot x + e^{bt}\cdot y)\ d\mu(x)
	\end{equation*}
	\item
		if $b=0$ then $\varphi\in (E)_\beta,\ t\in\R$:
	\begin{equation*}
	\ P_{a,0,t}(\varphi)=\int\limits_{E^*}\varphi(\sqrt{2at}\cdot x+y)\ d\mu(x)
	\end{equation*}
\end{enumerate}
\end{theorem}
On an informal level the second case of the above theorem may be considered as a special case of the first one. Note that by the rules of l'Hôpital we have $\lim\limits_{b\rightarrow 0} (b-a)\frac{1-e^{2bt}}{b}= 2at$.
%
%
%
%
\section {Summary and Bibligraphical Notes}
%
The investigation of regular transformation groups in this manuscript can be  compared with \cite{Chung1}.There the authors first construct a two-parameter transformation group $G$ on the space of white noise test functions $(E)$ in which the adjoints of Kuo's Fourier and Kuo's Fourier-Mehler transforms are included. They show that the group $G$ is a two-dimensional complex Lie group whose infinitesimal generators are the Gross Laplacian $\Delta_G$ and the number operator $N$, and then find an explicit description of a differentiable one-parameter subgroup of $G$ whose infinitesimal generator is $a\Delta_G +bN$, which is identical with ours.\\
Using the unique representation of white noise spaces by generalized Wick tensors and corresponding renormalized operators we have presented a different way which leads to the more general result for the spaces $(E)_\beta$. Using convolution operators and the fact that differentiable one parameter transformation groups on nuclear Fréchet spaces are always regular, convergence problems could be solved.
%
%
\par\bigskip\noindent
{\bf Acknowledgment.} 
The authors would like to thank to Prof. H.-H. Kuo for his encouragement to publish this manuscript. W.~Bock would like to thank for the financial support from the FCT project PTDC/MAT-STA/1284/2012.
%

\end{document}